\documentclass{amsart}

\usepackage{amsmath,amssymb, graphicx}

\newtheorem{theorem}{Theorem}[section]
\newtheorem{lemma}[theorem]{Lemma}
\newtheorem{definition}[theorem]{Definition}

\newtheorem{remark}[theorem]{Remark}
\newtheorem{question}[theorem]{Question}

\numberwithin{equation}{section}

\newcommand{\mN}{\mathcal{N}}
\newcommand{\mNt}{\widetilde{\mathcal{N}}(SG_{n,2})}
\newcommand{\mNh}{\overline{\mathcal{N}}(SG_{n,2})}
\newcommand{\alphaodd}{\alpha_{\mbox{\small{odd}}}}
\newcommand{\betaodd}{\beta_{\mbox{\small{odd}}}}
\newcommand{\gammaodd}{\gamma_{\mbox{\small{odd}}}}

\newcommand{\alphaeven}{\alpha_{\mbox{\small{even}}}}
\newcommand{\betaeven}{\beta_{\mbox{\small{even}}}}
\newcommand{\gammaeven}{\gamma_{\mbox{\small{even}}}}

\begin{document}

\title[Deformation Retracts of Neighborhood Complexes\ldots]{Deformation Retracts of Neighborhood Complexes of Stable Kneser Graphs}


\author{Benjamin Braun}
\address{Department of Mathematics, University of Kentucky, Lexington, Kentucky 40506}
\email{benjamin.braun@uky.edu}

\author{Matthew Zeckner}
\address{Department of Mathematics, University of Kentucky, Lexington, Kentucky 40506}
\email{matthew.zeckner@uky.edu}
%
\subjclass[2010]{Primary: 05E45  Secondary: 57M15, 05E18, 05C15}

\date{February 9, 2011}

\thanks{The first author was partially supported through NSF award DMS-0758321.  The second author was partially supported by a graduate fellowship through NSF award DMS-0758321.}

\keywords{Stable Kneser graph, neighborhood complex, discrete Morse theory, polytope}

\begin{abstract}
In 2003, A. Bj\"{o}rner and M. de Longueville proved that the neighborhood complex of the stable Kneser graph $SG_{n,k}$ is homotopy equivalent to a $k$-sphere.
Further, for $n=2$ they showed that the neighborhood complex deformation retracts to a subcomplex isomorphic to the associahedron.
They went on to ask whether or not, for all $n$ and $k$, the neighborhood complex of $SG_{n,k}$ contains as a deformation retract the boundary complex of a simplicial polytope.

Our purpose is to give a positive answer to this question in the case $k=2$.
We also find in this case that, after partially subdividing the neighborhood complex, the resulting complex deformation retracts onto a subcomplex arising as a polyhedral boundary sphere that is invariant under the action induced by the automorphism group of $SG_{n,2}$.
\end{abstract}

\maketitle


\section{Introduction and Main Result}

In 1978, L. Lov\'{a}sz proved in \cite{LovaszChromaticNumberHomotopy} M. Kneser's conjecture that if one partitions all the subsets of size $n$ of a $(2n + k)$-element set into $(k+1)$ classes, then one of the classes must contain two disjoint subsets.
Lov\'{a}sz proved this conjecture by modeling the problem as a graph coloring problem: see Section~\ref{Background} for definitions of the following objects.
For the Kneser graphs $KG_{n,k}$, Kneser's conjecture is equivalent to the statement that the chromatic number of $KG_{n,k}$ is equal to $k+2$.
Lov\'{a}sz's proof methods actually provided a general lower bound on the chromatic number of any graph $G$ as a function of the topological connectivity of an associated simplicial complex called the neighborhood complex of $G$.
Of particular interest in his proof was the critical role played by the Borsuk-Ulam theorem.
Later that year, A. Schrijver identified in \cite{Schrijvergraphs} a vertex-critical family of subgraphs of the Kneser graphs called the stable Kneser graphs $SG_{n,k}$, or Schrijver graphs, and determined that the chromatic number of $SG_{n,k}$ is equal to $k+2$.

In 2003, A. Bj\"{o}rner and M. de Longueville gave in \cite{BjornerDeLongueville} a new proof of Schrijver's result by applying Lov\'{a}sz's method to the stable Kneser graphs; in particular, they proved that the neighborhood complex of $SG_{n,k}$ is homotopy equivalent to a $k$-sphere.
In the final section of their paper, Bj\"{o}rner and De Longueville showed that the neighborhood complex of $SG_{2,k}$ contains the boundary complex of a $(k+1)$-dimensional associahedron as a deformation retract.
Their paper concluded with the following:

\begin{question}\label{BDQ}{\rm (Bj\"{o}rner and De Longueville, \cite{BjornerDeLongueville}) } For all $n$ and $k$, does the neighborhood complex of $SG_{n,k}$ contain as a deformation retract the boundary complex of a simplicial polytope?
\end{question}

Our main contribution in this paper is to provide a positive answer to Question~\ref{BDQ} in the case $k=2$.  Specifically, we show the following:

\begin{theorem}\label{mainthm}
For every $n\geq 1$, the neighborhood complex of $SG_{n,2}$ simplicially collapses onto a subcomplex arising as the boundary of a three-dimensional simplicial polytope.
\end{theorem}

The subcomplex for $\mN(SG_{3,2})$ is shown in Figure~\ref{p}.

\begin{figure}[ht]
\label{p}
\begin{center}
\includegraphics[width=5in]{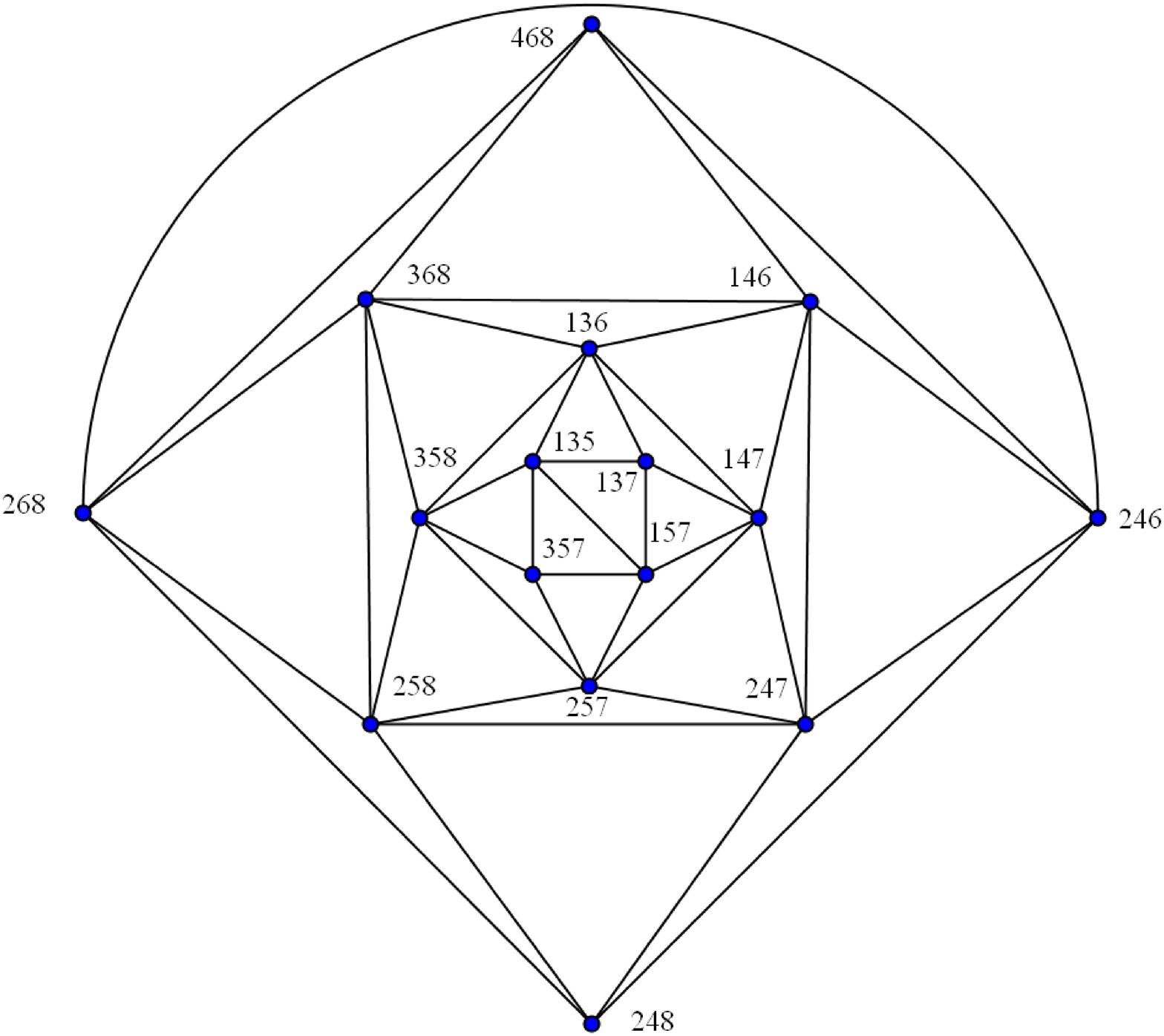}
\end{center}
\caption{}
\end{figure}

In \cite{BraunAutStableKneser}, the first author proved that for $k\geq 1$ and $n\geq 1$ the automorphism group of $SG_{n,k}$ is isomorphic to the dihedral group of order $2(2n+k)$.
It is natural to ask if there exist spherical subcomplexes of the neighborhood complex of $SG_{n,k}$ that are invariant under the induced action of this group.
While our spheres arising in Theorem~\ref{mainthm} are not invariant, we are able to show the following:

\begin{theorem}\label{dihedralthm}
For every $n\geq 1$, there exists a partial subdivision of the neighborhood complex of $SG_{n,2}$ that simplicially collapses onto a subcomplex invariant under the action induced by the automorphism group of $SG_{n,2}$ arising as the boundary of a three-dimensional simplicial polytope.
\end{theorem}

In addition to its aesthetic attraction, there are two primary reasons we are interested in Question~\ref{BDQ}.
First, any polytopes found in response to Question~\ref{BDQ} will be common generalizations of simplices, associahedra, and $1$-spheres given as odd cycles, due to the following observations: for $SG_{1,k}=K_{k+2}$, the neighborhood complex is a simplex boundary; for $SG_{n,1}$, the neighborhood complex is an odd cycle, hence a one-dimensional sphere; for $SG_{2,k}$, the neighborhood complex deformation retracts to an associahedron.
A family of polytopes generalizing these objects would be interesting to identify.
Second, a broad extension of the neighborhood complex construction is the graph homomorphism complex $HOM(H,G)$ studied in \cite{BabsonKozlovComplexes,BabsonKozlovLovaszConjecture,DochtermannEngstromCellular,DochtermannSchultz,SchultzStableKneserNotTest,SchultzSpacesOfCircuits}.
The complex $HOM(K_2,G)$ is known to be homotopy equivalent to the neighborhood complex of $G$.
The homomorphism complex construction leads to interesting phenomena, yet at present the lower bounds on graph chromatic numbers obtained by these are no better than those provided by the neighborhood complex.
We believe it is appropriate to continue to focus attention on the neighborhood complex construction along with the $HOM$ construction.

The rest of this paper is as follows.
In Section~\ref{Background}, we introduce the necessary background and notation regarding neighborhood complexes and stable Kneser graphs as well as discrete Morse theory, the primary tool in our proofs.
In Sections~\ref{Matching} and~\ref{Sphere1}, we provide a proof of Theorem~\ref{mainthm}.
In Section~\ref{Sphere2} we provide a proof of Theorem~\ref{dihedralthm}.


\section{Definitions and Background}\label{Background}

Let $[n]:=\{1,2,\ldots,n\}$.
The material in this section is adapted from the texts \cite{JonssonBook} and \cite{KozlovBook}, where more details may be found.


\subsection{Neighborhood Complexes and Stable Kneser Graphs}

The following definition is due to Lov\'{a}sz.

\begin{definition}
Given a graph $G=(V,E)$, the \emph{neighborhood complex of $G$} is the simplicial complex $\mN(G)$ with vertex set $V$ and faces given by subsets of $V$ sharing a common neighbor in $G$, i.e. $\mN(G) := \{ F\subset V: \exists v\in V \textrm{ s.t. } \forall u\in F, \{u,v\}\in E\}$.
\end{definition}

The graphs we are interested in are the following.

\begin{definition}\label{KG}
For $n\ge 1$ and $k\ge 0$ the \emph{Kneser graph}, denoted $KG_{n, k}$, is the graph whose vertices are the subsets of $[2n +k]$ of size $n$.
We connect two such vertices with an edge when they are disjoint as sets.

We call an $n$-set $\alpha$ of $[2n+k]$ \emph{stable} if $\alpha$ does not contain the subset $\{1, 2n+k\}$ or any of the subsets $\{i, i+1\}$ for $i=1,\ldots , 2n+k-1$.
The \emph{stable Kneser graph}, denoted $SG_{n, k}$, is the induced subgraph of $KG_{n, k}$ whose vertices are the stable subsets of $[2n+k]$.
\end{definition}

Our focus in this paper is on the case $k=2$; we will assume through the rest of the paper that this holds.
In order to handle different stable $n$-sets, we distinguish between them as follows, with all addition on elements being modulo $2n + 2$.

\begin{definition}
We call a stable $n$-set $\alpha$ \emph{tight} if $\alpha = \{i, i + 2, i + 4, \ldots, i + 2(n-1) \}$ for some $i \in [2n + 2]$.
Otherwise, we call $\alpha$ a \emph{loose} stable $n$-set.

For $\alpha=\{\alpha_1, \ldots, \alpha_n\}$ and $\beta$ stable $n$-sets, we call $\alpha$ and $\beta$ \emph{immediate neighbors} if $\alpha \oplus 1 = \beta$ or $\alpha \ominus 1 = \beta$, where $\alpha \oplus j:= \{\alpha_1 + j, \ldots, \alpha_n + j\}$ and $\alpha \ominus j$ is defined similarly.

We call $\alpha$ and $\beta$ \emph{outer neighbors} if there is an ordering of the elements of $\alpha$ such that $\beta =(\alpha_1 + 1, \alpha_2 + 1, \ldots, \alpha_{i-1} + 1 , \alpha_i + 2, \alpha_{i+1} + 1, \ldots, \alpha_n + 1)$ and $\alpha$ and $\beta$ are neighbors in $SG_{n, 2}$.
\end{definition}

The following remarks provide some insight into the structure of these graphs; further discussion, including proofs of these remarks, can be found in \cite{BraunIndComplexKneser,BraunAutStableKneser}.
\begin{itemize}
\item
A cycle is formed in $SG_{n,2}$ with vertices a stable $n$-set $\alpha$ and the stable $n$-sets $\alpha\oplus1$, $\alpha \oplus 2$, etc, with the edges $\{\alpha\oplus i,\alpha\oplus (i+1)\}$.
Thus, $\alpha$ and $\beta$ are immediate neighbors if they are neighbors on such a cycle in $SG_{n, 2}$.

\item
Stable $n$-sets $\alpha$ and $\beta$ are outer neighbors in $SG_{n, 2}$ if they
are neighbors and lie on two different cycles created via the immediate neighbor process.

\item
A loose stable $n$-set has degree $4$ in $SG_{n,2}$.  Two of its neighbors are immediate neighbors while the other two are outer neighbors.

\item
The tight stable $n$-sets correspond to vertices that together induce a complete bipartite subgraph in $SG_{n,2}$.

\end{itemize}


\subsection{Discrete Morse Theory}

We now introduce some tools from discrete Morse theory.
Discrete Morse theory was first developed by R. Forman in \cite{FormanMorseTheory} and has since become a powerful tool for topological combinatorialists.
The main idea of the theory is to systematically pair off faces within a simplicial complex in such a way that we obtain a collapsing order for the complex, yielding a homotopy equivalent cell complex.

\begin{definition}

A \textit{partial matching} in a poset $P$ is a partial matching in the underlying graph of the Hasse diagram
of $P$, i.e., it is a subset $M\subseteq P \times P$ such that
\begin{itemize}
\item
$(a,b)\in M$ implies $b \succ a;$ i.e. $a<b$ and no $c$ satisifies $a<c<b$.
\item
each $a\in P$ belongs to at most one element in $M$.
\end{itemize}
When $(a,b) \in M$, we write $a=d(b)$ and $b=u(a)$.
\item
A partial matching on $P$ is called \emph{acyclic} if there does not exist a cycle
\[
a_1 \prec u(a_1) \succ a_2 \prec u(a_2) \succ \cdots \prec u(a_m) \succ a_1
\]
with $m\ge 2$ and all $a_i\in P$ being distinct.

\end{definition}

Given an acyclic partial matching $M$ on a poset $P$, we call an element $c$ \emph{critical} if it is unmatched.
If every element is matched by $M$, we say $M$ is \emph{perfect}.
We are now able to state the main theorem of discrete Morse theory.

\begin{theorem}
Let $\Delta$ be a simplicial complex and let $M$ be an acyclic matching on the face poset of $\Delta$.
Let $c_i$ denote the number of critical $i$-dimensional cells of $\Delta$.
The space $\Delta$ is homotopy equivalent to a cell complex $\Delta_c$ with $c_i$ cells of dimension $i$
for each $i\ge 0$, plus a single $0$-dimensional cell in the case where the emptyset is paired in the matching.
\end{theorem}

\begin{remark}
If the critical cells of an acyclic matching on $\Delta$ form a subcomplex $\Gamma$ of $\Delta$, then $\Delta$ simplicially collapses to $\Gamma$, implying that $\Gamma$ is a deformation retract of $\Delta$.
\end{remark}

It is often useful to create acyclic partial matchings on several different sections of the face poset of a simplicial complex and then combine
them to form a larger acyclic partial matching on the entire poset.
This process is detailed in the following theorem known as the \textit{Cluster lemma} in \cite{JonssonBook} and the \textit{Patchwork theorem} in \cite{KozlovBook}.

\begin{theorem}\label{patchwork}
Assume that $\varphi : P \rightarrow Q$ is an order-preserving map.
For any collection of acyclic matchings on the subposets $\varphi^{-1}(q)$ for $q\in Q$, the union of these matchings is itself an acyclic matching on $P$.
\end{theorem}


\section{Construction of the acyclic matching}\label{Matching}

In this section we use discrete Morse theory to describe a simplicial collapsing of $\mN(SG_{n,2})$.
Section~\ref{Sphere1} contains an analysis of the complex of critical cells of our discrete Morse matching.
Our approach will be to produce poset maps from subposets of the face poset of $\mN(SG_{n,2})$ to various target posets, construct acyclic matchings on inverse images of these poset maps, and apply Theorem~\ref{patchwork} to obtain an acyclic matching on the entire face poset.
In our construction of these poset maps, we will consider facets of $\mN(SG_{n,2})$, which by definition arise in the following way.

\begin{definition}
For $\gamma$ a vertex of $SG_{n,2}$, let $\Sigma_{\gamma}$ be the facet in $\mN(SG_{n,2})$ formed by the neighbors of $\gamma$.
\end{definition}

A key role in our simplicial collapsing is played by the two simplices formed by the collections of all vertices of $SG_{n,2}$ of the form $\{\alpha_1,\ldots,\alpha_n\}$, where in each simplex the $\alpha_i$ have all even or all odd entries, respectively.
These all even and all odd simplices may be viewed as North and South poles for the complex.
As these pole simplices are not two-dimensional, we must collapse them to smaller dimension.
The facets $\Sigma_{\gamma}$ where $\gamma$ is loose then collapse to pairs of triangles that interpolate between these two poles, forming our sphere.


\subsection{Collapsing in facets of loose stable $n$-sets}

For any loose stable $n$-set $\alpha$, $\Sigma_{\alpha}$ is a $3$-dimensional simplex in $\mN(SG_{n,2})$ formed by the outer and immediate neighbors of $\alpha$.
A routine check reveals that the edge consisting of $\alpha$'s outer neighbors is free in $\mN(SG_{n,2})$.
Thus, for each such facet we may perform the following collapse.

Label the vertices of $\Sigma_\alpha$ by $a,b,c,d$ with the outer
neighbors of $\alpha$ labeled $b$ and $d$. Let $P_\alpha$ be the face poset of $\Sigma_{\alpha}$ and
$Q_\alpha:=A<B_\alpha$ a chain of length $2$.  Let $\theta_\alpha : P_\alpha \rightarrow Q_\alpha$ be defined by
\[
\theta_\alpha (x) = \left \{ \begin{array}{ll}
A& \mbox{if } \{ b, d\} \nsubseteq x \\
B_\alpha & \mbox{if } \{ b, d\} \subseteq x
\end{array}
\right.
\]
It is immediate that $\theta_\alpha$ is a poset map and that $\theta^{-1}(B_\alpha)$ yields a perfect acyclic matching when we match an element $x$ in the inverse not containing $a$ with $x \cup \{a\}$.
This matching collapses each facet given by a loose stable $n$-set to two triangles that share a common edge.


\subsection{Collapsing in facets of tight stable $n$-sets}\label{collapse}

Consider a tight stable $n$-set $\alpha$ in $[2n+2]$, and observe that all elements of $\alpha$ are of the same parity.
\begin{lemma}
$\alpha$ has a unique outer neighbor.
\end{lemma}

\begin{proof}
Observe that $[2n+2] \setminus \alpha$ consists of the $n+1$ elements of the opposite parity of the elements of $\alpha$ and the one remaining element of the same parity as the elements of $\alpha$.
An outer neighbor of $\alpha$ must contain the one element of the same parity as the elements of $\alpha$, which we denote $p$.
As the outer neighbor is a stable $n$-set, it cannot contain $p\pm 1$.
Since there are only  $n-1$ viable elements left in $[2n+2] \setminus \alpha$, an outer neighbor of $\alpha$ must contain them all.
Hence, $\alpha$ has a unique outer neighbor.
\end{proof}

To simplify our presentation we introduce additional notation.
Lexicographically assign the neighbors of $\alpha$ the labels $ v^1, v^2, \ldots, v^{n+1}$ and $\eta_{\alpha}$
where the $v^i$'s are all tight and of the opposite parity of $\alpha$.
The remaining vertex, $\eta_\alpha$, denotes $\alpha$'s unique outer neighbor.

Let $\Sigma_{\alpha}$ denote the $(n+1)$-simplex formed by the neighbors of $\alpha$ and let $P_\alpha$ denote the face poset of $\Sigma_\alpha$.
Given $\alpha$ and its unique outer neighbor $\eta_\alpha$, let $p$ denote the element in the outer neighbor $\eta_\alpha$ of identical parity to the elements of $\alpha$.
For some $j$, we obtain $v^j$ and $v^{j+1}$ from $\eta_\alpha$ by replacing $p$ with $p-1$ or $p+1$, respectively.

\begin{lemma}
$\Sigma_\alpha$ collapses to the simplicial complex $N_\alpha$ where $N_\alpha$ consists of the following facets and their subsets:
\[
\{v^1,v^2,v^3\} , \{v^1,v^3,v^4\}, \{v^1,v^4,v^5 \}, \ldots , \{v^1, v^n, v^{n+1} \}, \{v^j, v^{j+1}, \eta_\alpha \}
\]
where if $j= n + 1$ then the last set listed above is replaced by $\{ v^1 , v^{n+1} , \eta_\alpha\}$.
In other words, $\Sigma_\alpha$ collapses to a triangulated $(n+1)$-gon where all diagonals in the triangulation emanate from the vertex labeled $v^1$ and the triangle $\{v^j, v^{j+1}, \eta_\alpha \}$ is attached to the $(n+1)$-gon.
\end{lemma}

The idea behind our matching in the following proof is that the intersection of any two facets corresponding to tight sets is the simplex $\{v^1,v^2,v^3,\ldots,v^{n+1}\}$.
To collapse $\Sigma_\alpha$ to $N_\alpha$, we will pair unwanted faces contained in $\{v^1,v^2,v^3,\ldots,v^{n+1}\}$ with $v^1$, and pair unwanted faces containing $\eta_\alpha$ with $v^j$.
Separating these matchings allows us to patch the relevant poset maps together in a coherent way in the following subsection.

\begin{proof}
Fix a tight stable $n$-set $\alpha$, with outer neighbor $\eta_\alpha$ and associated $v^j$.
Let $Q_\alpha:=A<B<C_\alpha$ be a three element chain.
Consider the map $\varphi_\alpha : P_\alpha \rightarrow Q_\alpha$ defined by
\[
\varphi_\alpha (x) = \left \{ \begin{array}{ll}
A & \mbox{if } |x|=1, \mbox{ } x= \{v^r, v^s\}, x= \{v^1, v^r, v^s\}, \mbox{ or } x\subseteq\{v^j,v^{j+1},\eta_\alpha\}\\
B & \mbox{for all other } x \mbox{ such that } \eta_\alpha \notin x \\
C_\alpha & \mbox{otherwise}     \\
\end{array}
\right.
\]
where either
$r =1$ and $s\in [n+1]\setminus \{1\}$ or $s=r+1$ for $r\in [n]\setminus \{1\}$.
Observe that $\varphi^{-1}_\alpha(A)$ is exactly the complex $N_\alpha$ defined above.

We now construct acyclic matchings on the posets $\varphi_\alpha^{-1}(B)$ and $\varphi_\alpha^{-1}(C_\alpha)$.
We claim that matching each $x\in \varphi_{\alpha}^{-1}(B)$ not containing $v^1$ with $x \cup \{v^1\}$ yields a perfect acyclic matching.
One first needs to check that no element is paired with an element of $\varphi_{\alpha}^{-1}(A)$ or $\varphi_{\alpha}^{-1}(C_\alpha)$, which is clear from the definitions.
That every face is matched is similarly clear.
To verify acyclicity, suppose a cycle exists, say $x_1 \prec u(x_1) \succ x_2 \prec u(x_2) \succ \cdots \prec u(x_m) \succ x_1$, for $m$ minimal.
Then, both $u(x_1)$ and $u(x_m)$ contain $x_1$ (as sets).
However, our matching dictates that we match $x_1$ and $u(x_1)$ if and only if they are in $\varphi_\alpha^{-1} (B)$ and
$u(x_1) = x_1 \cup \{v^1\}$.
If $v^1\in u(x_m)$, then $u(x_m) = u(x_1)$ implying $x_m = x_1$, a contradiction.
Otherwise, $v^1\notin u(x_m)$ implies $u(x_m)$ is also matched with $u(x_m) \cup \{v^1\}$, a contradiction.

We claim that matching each $x\in \varphi_{\alpha}^{-1}(C_\alpha)$ such that $j\notin x$ with $x\cup \{j\}$ yields a perfect acyclic matching in $\varphi_{\alpha}^{-1}(C_\alpha)$.
It is clear from the definitions that no element is paired with something outside $\varphi_{\alpha}^{-1}(C_\alpha)$, keeping in mind the observation that the pairs $(\eta_\alpha,\{v^j,\eta_\alpha\})$ and
$(\{v^{j+1},\eta_\alpha\},\{v^j,v^{j+1},\eta_\alpha\})$ are not included in this preimage; they are included in the preimage $\varphi_\alpha^{-1}(A)$.
Verifying that this is a perfect acyclic matching is similar to the previous case.
\end{proof}


\subsection{Combining the loose and tight cases to form a single poset map}

Our matchings were all defined by studying poset maps with domains the facets of $\mN(SG_{n,2})$.  To apply Theorem~\ref{patchwork}, we need to show that these maps may be combined into a single poset map in a coherent manner.
Consider the poset $Q(n,2)$ formed by identifying along commonly named elements the posets $Q_\alpha$ from the constructions of our matchings.
In other words, $Q(n,2)$ has a unique minimal element $A$, a maximal chain on two vertices labeled $A<B_\alpha$ for each loose $n$-set $\alpha$, and a maximal chain of length three labeled $A<B<C_\alpha$ for each tight $n$-set $\alpha$ that all share the common subchain $A<B$.
For each of the poset maps $\theta_{\alpha}$ and $\varphi_{\alpha}$ defined in the previous subsection, we view them as a map from $P_\alpha$ to $Q(n,2)$.

Let $P(n, 2)$ denote the face poset of $\mN(SG_{n,2})$.
We define a map $\Phi$ from $P(n,2)$ to $Q(n,2)$  by mapping a face $x\in \Sigma_{\alpha}$ to
\[
\Phi(x) = \left\{ \begin{array}{ll}
\theta_{\alpha}(x) &  \mbox{if } \alpha \mbox{ is loose} \\
\varphi_{\alpha}(x) &  \mbox{if } \alpha \mbox{ is tight } \\
\end{array}
\right.
\]

\begin{lemma}
$\Phi$ is a well-defined poset map.
\end{lemma}

\begin{proof}
Assuming that $\Phi$ is well-defined, that it is a poset map is immediate since $\theta$ and $\varphi$ are poset maps.
To verify $\Phi$ is well-defined, we need to check that faces contained in more than one facet are mapped coherently by $\Phi$.
Let $\alpha^1$ and $\alpha^2$ be two stable sets that yield the facets $\Sigma_{\alpha^1}$ and $\Sigma_{\alpha^2}$ in $\mN(SG_{n,2})$.

\underline{Case 1:}
Suppose $\alpha^1$ and $\alpha^2$ are both loose sets.
We consider the size of the intersection of their respective facets.
If $|\Sigma_{\alpha^1} \cap \Sigma_{\alpha^2}| = 4$, then $\alpha^1 = \alpha^2$ and we are done.
Suppose $|\Sigma_{\alpha^1} \cap \Sigma_{\alpha^2}| = 3$.
Say $\{v^1, v^2, v^3\} \subset \Sigma_{\alpha^1} \cap \Sigma_{\alpha^2}$ along with all their subsets for some vertices $v^1, v^2,$ and $v^3$.
Consider the support of these vertices, $supp(v^1, v^2, v^3)$, where $supp(v^1,\ldots,v^k):=\cup_i v^i$ as sets.
We know each of these vertices avoid the stable $n$-sets $\alpha^1$ and $\alpha^2$, thus there are at most $n+1$ viable elements remaining in $[2n+2]$.
However, $|supp(v^1, v^2, v^3)| \ge n+2$, since the intersection of any two of the vertices can have at most $n-1$ elements in common.
Thus, this case does not occur.
Suppose $|\Sigma_{\alpha^1} \cap \Sigma_{\alpha^2}| \le 2$.
In this case, $\Sigma_{\alpha^1} \cap \Sigma_{\alpha^2}$ is either a single vertex or an edge between an inner and an outer neighbor.
As any such face is sent to $A$ by both $\theta_{\alpha^1}$ and $\theta_{\alpha^2}$, we see that $\Phi$ is well-defined on the intersection of pairs of loose sets.

\underline{Case 2:}
If $\alpha^1$ and $\alpha^2$ are both tight sets with $\alpha^1 \ne \alpha^2$, then $|\alpha^1 \cap \alpha^2|= n-1$, implying that $\Sigma_{\alpha^1} \cap \Sigma_{\alpha^2}$ is an $n$-dimensional simplex $\Sigma$.
Using our previous notation, $\Sigma=\{v^1,v^2,v^3,\ldots,v^{n+1}\}$.
For a given face $x\in \Sigma$, every map $\varphi_\alpha$ maps $x$ to either $A$ or $B$ in a coherent manner, as the definitions of the $\varphi$-maps are the same on $\Sigma$.
Thus, $\Phi$ is well-defined on the intersection of pairs of tight sets.

\underline{Case 3:}
Suppose $\alpha^1$ is a tight set and $\alpha^2$ is a loose set.
If $|\alpha^1 \cap \alpha^2| \le n-2$, then $|supp(\alpha^1,\alpha^2)| \ge n+2$.
Hence, $|[2n+2] \setminus supp(\alpha^1,\alpha^2)|\le n$.
Thus, $\Sigma_{\alpha^1}$ and $\Sigma_{\alpha^2}$ intersect in a vertex $x$, and $\Phi (x) =A$ is well-defined.

If $|\alpha^1 \cap \alpha^2| = n-1$ consider $F=[2n+2]\setminus supp(\alpha^1, \alpha^2)$.
We know $|F|=n+1$ as $|supp(\alpha^1, \alpha^2)|=n+1$.
Moreover, $F$ consists of $n$ elements of the opposite parity of $\alpha^1$ and one element, say $p$, of the same parity.
From this we know that $p \pm 1$ is in $F$, but not both.
Hence, $F$ contains only two stable $n$-sets, one set $\beta$ which is tight and whose elements
are of opposite parity of $\alpha^1$
and another set $\gamma$, which consists of $p$ and $n-1$ elements of opposite parity of $\alpha^1$ not including $p\pm 1$.
Thus, $\Sigma_{\alpha^1} \cap \Sigma_{\alpha^2} = \{\beta, \gamma\}$, all faces of which are mapped to $A$ by $\varphi_{\alpha^1}$.

We next show that all faces of $F$ are mapped to $A$ by $\theta_{\alpha^2}$ as well.
As $\alpha^2$ is loose and $\beta$ is a tight neighbor, we know that $\beta$ is an outer neighbor of $\alpha^2$.
In addition, $\gamma$ is an immediate neighbor of $\alpha^2$ by construction.
The only edge sent to $B_{\alpha^2}$ by $\theta_{\alpha^2}$ is the one formed by both outer neighbors of $\alpha^2$, which this edge is not, thus it is sent to $A$.
Hence, $\Phi$ is well-defined.

\end{proof}

To use Theorem~\ref{patchwork}, we now need to verify that our previous matchings are valid.
Recall that
\[\Phi^{-1}(A)=\cup_\alpha\varphi^{-1}_\alpha(A) \bigcup \cup_\alpha\theta_\alpha^{-1}(A),\]
and that none of these preimages carried matchings.
The other preimages of $\Phi$ correspond to the preimages of $\theta_\alpha$ or $\varphi_\alpha$ depending whether $\alpha$ is loose or tight.
Our previous matchings may therefore be applied, after noting that on $\Phi^{-1}(B)$ the matching is independent of choice of $\alpha$.
Hence by Theorem~\ref{patchwork} and the remark following it we have that $\mN(SG_{n,2})$ simplicially collapses onto the complex whose face poset is $\Phi^{-1}(A)$.


\section{Analysis of the complex of critical faces}\label{Sphere1}

Throughout this section it will be useful to refer to Figure~\ref{p}, illustrating the case $n=3$.
Denote by $\mNt$ the complex of critical faces given by $\Phi^{-1}(A)$.
By construction, $\mNt$ is two-dimensional and pure;  in this section we prove that it is the boundary of a three-dimensional simplicial polytope.
Our approach is to first construct a planar graph inducing a triangulation of $S^2$ that realizes $\mNt$, then to apply the following theorem.
Recall that a graph $G$ is \emph{$3$-connected} if for any pair of vertices $v$ and $w$ in $G$, there exist three disjoint paths from $v$ to $w$.

\begin{theorem}\label{Steinitz}{\rm (Steinitz' theorem, see \cite{ZieglerLectures})} A simple graph $G$ is the one-skeleton of a three-dimensional polytope if and only if it is planar and $3$-connected.
\end{theorem}


\subsection{Construction of $\mNt$}

We want to realize $\mNt$ as a triangulation of $S^2$; we will do so by constructing its one-skeleton in the plane.
We begin with notation and several lemmas.
For a stable $n$-set $\alpha$, let  $\alpha_{\mbox{\small{odd}}}$ be the set of all odd elements of $\alpha$ and let $\alpha_{\mbox{\small{even}}}$ be the set of all even elements of $\alpha$.
Throughout this subsection, unless otherwise indicated, we assume for a stable $n$-set $\alpha=\{\alpha_1, \ldots , \alpha_n\}$ that $\alphaodd =\{\alpha_1, \ldots, \alpha_i\}$ and $\alphaeven = \{\alpha_{i+1}, \ldots, \alpha_n\}$.
Let $P_i$ denote the set of stable $n$-sets consisting of $i$ even elements and $n-i$ odd elements.

\begin{lemma}
$P_i=\{\alpha^0,\ldots,\alpha^n\}$ is lexicographically ordered by setting \[\alpha^0:=\{1,3,\ldots,2(n-i)-1,2(n-i)+2,\ldots,2n\}\] and $\alpha^j:=\alpha^0\ominus 2j$.   Also, $\alpha^n \ominus 2 = \alpha^0$.
\end{lemma}

\begin{proof}
For $P_0$, it is immediate that $\{1,3,5,\ldots,2n-1\}\leq \{1,3,5,\ldots,2n-3,2n+1\} \leq \{1,3,5,\ldots,2n-5,2n-1,2n+1\}\leq \cdots$
orders $P_0$ lexicographically.
Given an $\alpha\in P_0$, it follows by inspection that $\alpha\ominus 2$ is the next term in the sequence.
The set $P_n$ is handled similarly.

For the case of $P_i$, $0<i<n$, it is immediate that $\{1,3,5,\ldots,2(n-i)-1,2(n-i)+2,\ldots,2n\}\leq \{1,3,5,\ldots,2(n-i)-3,2(n-i),\ldots,2n,2n+1\}\leq \cdots \leq \{3,5,7,\ldots,2(n-i)+1,2(n-i)+4,\ldots,2n+2\}$
orders $P_i$ lexicographically.
Given an $\alpha\in P_i$, it follows by inspection that $\alpha\ominus 2$ is the next term in the sequence.
\end{proof}

\begin{lemma}\label{neigbors_cycle}
If $\alpha \in P_i$ and $\beta, \gamma \in P_{i+1}$ such that $|\alpha \cap \beta| = |(\alpha\ominus 2) \cap \beta| = n-1$,
$|(\alpha\ominus 2)  \cap \beta| = |(\alpha\ominus 2) \cap \gamma| = n-1$,
and $|(\alpha\ominus 2) \cap (\beta \ominus 2)| = |(\alpha\ominus 4) \cap \beta\ominus 2| = n-1$, then $\gamma=\beta\ominus 2$.
\end{lemma}

\begin{proof}
We maintain our ordering of the elements for a stable $n$-set $\alpha=\{\alpha_1, \ldots , \alpha_n\}$ as $\alphaodd =\{\alpha_1, \ldots, \alpha_i\}$ and $\alphaeven = \{\alpha_{i+1}, \ldots, \alpha_n\}$.
As $\alpha \in P_i$ we have
\begin{eqnarray}
\alphaodd & = & \{\alpha_1, \ldots, \alpha_i\} \nonumber \\
(\alpha \ominus 2)_{\mbox{\small{odd}}} & = &  \{\alpha_1-2, \ldots, \alpha_i - 2\}\nonumber\\
& = &  \{\alpha_1, \ldots, \alpha_{i-1}, \alpha_n +1\}\nonumber\\
(\alpha \ominus 4)_{\mbox{\small{odd}}} & = &  \{\alpha_1-4, \ldots, \alpha_i - 4\}\nonumber\\
& = &  \{\alpha_1, \ldots, \alpha_{i-2}, \alpha_n -1 , \alpha_n + 1\}\nonumber
\end{eqnarray}
By our assumptions about $\beta$ and $\gamma$ we have
\begin{eqnarray}
\betaodd &=& \alphaodd \cap (\alpha \ominus 2)_{\mbox{\small{odd}}} = \{\alpha_1, \ldots, \alpha_{i-1}\}\nonumber\\
\gammaodd &=& (\alpha \ominus 2)_{\mbox{\small{odd}}} \cap (\alpha \ominus 4)_{\mbox{\small{odd}}} = \{\alpha_1, \ldots, \alpha_{i-2}, \alpha_n +1\}\nonumber
\end{eqnarray}
Thus, $\betaodd \ominus 2 = \gammaodd$.
By a similar argument we see that $\betaeven \ominus 2 = \gammaeven$ and hence $\beta \ominus 2 =\gamma$.
So $\beta$ and $\gamma$ are neighbors in $P_{i+1}$.

\end{proof}

To construct our planar graph, order the elements of $P_0$ lexicographically and denote them $v^0, \ldots , v^n$.
Draw a regular $(n+1)$-gon, which we will also refer to as $P_0$, and cyclically label its vertices by $v^0, \ldots , v^n$.
Triangulate $P_0$ so that each diagonal in the triangulation has the vertex $v^0$ as an endpoint.
Next, we draw a second regular $(n+1)$-gon, denoted $P_1$, around $P_0$, satisfying two conditions:
\begin{itemize}
\item The vertices of $P_1$ lie outside $P_0$ on lines through the center point of $P_0$ and the midpoints of the edges of $P_0$, and
\item The edges of $P_1$ do not intersect $P_0$.
\end{itemize}
Label the vertices of the polygon $P_1$ by the elements of the set $P_1$, where the labels are placed cyclically about the circle in the lexicographic order; the lexicographically first label for $P_1$ is placed on the ray between the center of $P_0$ and the edge between $v^0$ and $v^1$.
Connect a vertex $v$ of $P_0$ to a vertex $w$ of $P_1$ if $|v\cap w| = n-1$, i.e. connect $w$ to the endpoints of the edge of $P_0$ that it is nearest to.

We inductively continue this process for $i\leq n$ by drawing an $(n+1)$-gon denoted $P_i$ around $P_{i-1}$.
Label the vertices of the polygon $P_i$ with the elements of the set $P_i$ in such a way that one may connect a vertex $v$ of $P_{i-1}$ to a vertex $w$ of $P_i$ exactly when $|v\cap w| = n-1$.
This results in the vertices of $P_i$ being labeled cyclically with respect to lexicographic order, with the requirement that the lex-first label for $P_i$ is placed on the ray between the center of $P_0$ and the edge between the lexicographically first and second vertices of $P_{i-1}$.
To complete the construction, once $P_n$ has been drawn and connected to $P_{n-1}$, draw arcs representing the edges $\{\{2,4, 6, \ldots , 2n \}, e \}$ for all even, tight stable $n$-sets $e$.
It is immediate from our lemmas that this construction is legitimate, and also it is clear that it yields a triangulation of the sphere.

To finish our proof, we must show that the facets of $\mNt$ are the same as the facets of this triangulation, i.e. that this triangulation is actually a realization of $\mNt$.
Observe that both $P_0$ and $P_n$ bound triangulated $(n+1)$-gons where the vertices of $P_0$ are the odd tight sets while the vertices of $P_n$ are the even tight sets and the diagonals in the triangulations all emanate from the lexicographically smallest tight stable set in each of $P_0$ and $P_n$.
These triangulated polygons correspond exactly to the triangulated polygons contained in the $N_\alpha$ complexes defined in regards to facets of tight $n$-sets.
What remains is to show that every other facet of our triangulation corresponds to a two-dimensional simplex in $\mNt$ and vice versa.

Let $\Sigma$ be a simplex in $\mNt$.  We will show that $\partial \Sigma$ exists in our constructed graph.
We consider three cases.

\underline{Case 1:}
Suppose $\Sigma$ consists of only tight vertices.  Then $\Sigma = \{v^1, v^j, v^{j+1}\}$ for some $j= 2, \ldots, n.$
As $v^j$ and $v^{j+1}$ are lexicographically ordered, we have $v^j= v^{j+1} \ominus 2 $
In the construction of our graph we cyclically connected vertices ordered lexicographically, hence the
edge $\{v^j, v^{j+1}\}$ exists in our graph.  Moreover, in the construction of our graph we connected
all vertices of $P_0$ to the vertex $\{1, 3, 5, \ldots, 2n-1\}$ and all vertices of $P_n$ to  the vertex $\{2,4, 6, \ldots , 2n \}$.
These are precisely the edges $\{v^1, v^j\}$ and $\{v^1, v^{j+1}\}$.

\underline{Case 2:} Suppose $\Sigma$ consists of tight and loose vertices.
This case follows easily from the following lemma.

\begin{lemma}\label{neighbors}
For $\alpha \in P_i$, there exists a unique vertex $\pi\in P_{i+1}$ such that $|\alpha\cap \pi|=|\alpha\ominus 2 \cap \pi| = n-1$.
\end{lemma}

\begin{proof}
Let $\alpha$, $\alpha \ominus 2 \in P_i$ be two neighboring vertices.  We consider two cases.

Suppose $\alpha$ and $\alpha \ominus 2$ are both tight sets.  Without loss of generality
we may assume $\alpha, (\alpha \ominus 2) \in P_0$ and we have $|\alpha \cap (\alpha \ominus 2)| = n-1$.  If $\pi$ is a common
neighbor to both $\alpha$ and $\alpha \ominus 2$ in $\mNt$ then, by construction, $\pi = \alpha \cap (\alpha \ominus 2) \cup \{p\}$ for some
$p \in [2n+2]$.  We claim $p$ must be $\alpha_n - 1$.  By definition, $p$ cannot be any element in
$\alpha\cup (\alpha \ominus 2)$.  There are $n+1$ such elements.  Additionally, $p$ cannot be any of the $n$ elements adjacent
to an element in $\alpha\cap (\alpha \ominus 2)$.  Thus, we are left with only one choice for $p$ as claimed.

Suppose $\alpha$ and $\alpha \ominus 2$ are both loose sets.
Set $\pi = (\alphaodd \cap (\alpha_{\mbox{\small{odd}}} \ominus 2)) \cup (\alphaeven \cup (\alpha_{\mbox{\small{even}}} \ominus 2))$.
From our definition of $\pi$ it is immediate that $\pi$ is a stable $n$-set.
Moreover, the definitions of $\alpha, (\alpha \ominus 2),$ and $\pi$ we have $|\alpha \cap \pi| = |(\alpha \ominus 2)\cap \pi| = n-1$.
Finally, as
$|\alphaodd \cap (\alpha_{\mbox{\small{odd}}} \ominus 2)| = i-1$ and
 $|\alphaeven \cup (\alpha_{\mbox{\small{even}}} \ominus 2)| = n -i + 1$ we have
that $\pi\in P_{i+1}$.  The uniqueness of $\pi$ follows from the definitions of $\alpha$ and $\alpha \ominus 2$.
Thus our claim holds.
\end{proof}

A similar argument shows that if $\alpha, (\alpha \ominus 2) \in P_i$, then there exists a unique vertex $\pi \in P_{i-1}$ that is a neighbor to both $\alpha$ and $(\alpha \ominus 2)$ for $i = 1, \ldots, n+1$, where $\pi= (\alphaeven \cap (\alpha_{\mbox{\small{even}}} \ominus 2)) \cup (\alphaodd \cup (\alpha_{\mbox{\small{odd}}} \ominus 2))$.

\underline{Case 3:}
Suppose $\Sigma$ consists of only loose vertices.
By construction of our poset map, we know that two of these vertices, $v^r$ and $v^s$, are immediate neighbors to some vertex $\alpha$
and the other vertex, $v^t$ is an outer neighbor of $\alpha$.
Set $\alpha =\{\alpha_1, \ldots, \alpha_n\}$ where $\alpha$ is a concatenation of $\alphaodd$ and $\alphaeven$.
Then, without loss of generality,
$\alpha \ominus 1 = v^r$ and $\alpha \oplus 1 = v^s$.  This implies that $v^s \ominus 2 = v^r$ which is an edge in our graph.
By definition of an outer neighbor,
$v^t = \{(\alpha_1 +1 , \alpha_2 +1, \ldots, \alpha_{j-1} +1 , \alpha_j \pm 2, \alpha_{j+1}+1 , \ldots, \alpha_n +1 )$ where
$\alpha_i$ is odd for $i=1, \ldots, j-1$ and is even for $i= j+1, \ldots, n$.  The parity of $\alpha_j$ is unknown.
If $\alpha_j$ is odd, then
$v^t = \{(\alpha_1 +1 , \alpha_2 +1, \ldots, \alpha_{j-1} +1 , \alpha_j + 2, \alpha_{j+1}+1 , \ldots, \alpha_n +1 )$.
From this we immediately see that $v^t \setminus v^s = \{\alpha_j + 2\}$, so $|v^s\cap v^t| = n-1$.
Consider $v^t \setminus v^r$.  We claim that $v^t\setminus v^r = \{\alpha_n +1\}$ implying
$|v^r\cap v^t| = n-1$ so that the edges $\{v^r,v^t\}$ and $\{v^s,v^t\}$ exist in our graph by
claim 2 of Lemma~\ref{neighbors}.

It is enough to show that
$\alpha_j + 2 \in v^r$ as we know $\alpha_n +1 \notin v^r$ and the remaining elements of $v^t$ are in $v^r$ by the definitions
of $v^t$ and $v^r$.
Since, by assumption, $\alpha_j$ is odd, we know that there is a gap of size two between $\alpha_j$ and $\alpha_{j+1}$.  
Now, $\alpha_{j+1} -1 \in v^r$ by definition and is odd.
As $\alpha_j$ is odd, it is also the case that $\alpha_j +2$ is odd.  
Moreover, there is only one odd number between $\alpha_j$ and $\alpha_{j+1}$.  
Thus $\alpha_j +2 = \alpha_{j+1} -1$.
The case when $\alpha_j$ is even follows similarly.

Now consider a simplex $\sigma$ in our constructed complex.  
If $\sigma$ consists of only tight vertices, then it is immediate from Case 1 that $\tau \in \mNt$, as we constructed it to be so.
If $\tau$ consists of any loose vertices then the fact that it is also in $\mNt$ is immediate from Lemma~\ref{neighbors} or Case 3 above.


\subsection{Proof that $\mNt$ is a simplicial polytope}\label{polytopeproof}

Let $G_n$ be the $1$-skeleton of $\mNt$.
By definition, $G_n$ is simple; the planarity of $G_n$ is shown by the our construction.
To apply Theorem~\ref{Steinitz} and complete our proof, we must show that $G_n$ is $3$-connected.

Let $x$ and $y$ be any two vertices of $G_n$.
We will show that there exist (at least) three disjoint paths from $x$ to $y$.
The above construction shows us that $G_n$ is built from $n+1$ concentric $(n+1)$-cycles, labeled from inside out $P_0, \ldots P_n$.
Recall that each vertex $v$ on a given cycle $P_i$, with the exception of the two cycles formed by tight vertices, is connected to two pairs of adjacent vertices off $P_i$, one pair on each of $P_{i-1}$ and $P_{i+1}$.
Each vertex $v$ on either $P_0$ or $P_n$ is connected to only one vertex on an adjacent cycle, either $P_1$ or $P_{n-1}$, respectively.

Suppose first that $x$ and $y$ lie on the same cycle $P_i$.
Traverse $P_i$ from $x$ to $y$ in opposite directions to obtain two edge-independent paths.
The third path can be found by first moving from $x$ to an adjacent cycle, $P_{i+1}$ or $P_{i-1}$, then traveling around this cycle in either direction until a neighbor of $y$ is reached.

Next, suppose $x$ and $y$ lie on different cycles, say $x$ on $P_j$ and $y$ on $P_k$ with $j<k$; we begin by finding a pair of disjoint paths from $x$ to $y$.
We first construct a pair of disjoint paths from $x$ to $P_k$.
Let $v^{1}$ and $w^{1}$ be the neighbors of $x$ that lie on $P_{j+1}$.
If $j+1=k$, stop at this point having constructed paths $x,v^1$ and $x,w^1$, otherwise proceed.
Let $r^2$ and $v^2$ be the neighbors of $v^{1}$ on $P_{j+2}$ and $v^2$ and $w^2$ be the neighbors of $w^1$ on $P_{j+2}$, noting that $v^2$ is a common neighbor of $v^1$ and $w^1$.
If $j+2=k$, stop at this point having constructed paths $x,v^1,v^2$ and $x,w^1,w^2$, otherwise proceed.

Now we are in the same situation with $v^2$ and $w^2$ as we were in with $v^1$ and $w^1$, in that we may denote the neighbors of $v^2$ on $P_{j+4}$ by $r^3$ and $v^3$ and the neighbors of $w^3$ by $v^3$ and $w^3$, which allows us to construct paths $x,v^1,v^2,v^3$ and $x,w^1,w^2,w^3$ from $x$ to $P_{j+3}$.
If $y$ is not on $P_{j+3}$, then as in the previous cases, we may extend these two paths by setting $v^4$ equal to the unique common neighbor of $v^3$ and $w^3$ on $P_{j+4}$ and setting $w^4$ equal to the other neighbor of $w^3$ on $P_{j+4}$.
We continue in this fashion, creating two paths that curve side-by-side through the graph $G_n$, until we reach $P_k$ with paths $x,v^1,\ldots,v^{k-j}$ and  $x,w^1,\ldots,w^{k-j}$.
Note that $v^{k-j}$ and $w^{k-j}$ are neighbors on $P_k$ by construction.
If $v^{k-j}$ and $w^{k-j}$ are both on $P_k$ and neither is $y$, then we may extend these two paths along $P_k$ in opposite directions until we meet $y$.
If either $v^{k-j}$ or $w^{k-j}$ is $y$, then we may complete the other path by connecting via one edge.

Having completed two disjoint paths from $x$ to $y$, we now need to find a third path disjoint from the first two.
If $j+1=k$, let $z$ be a neighbor of $y$ on $P_j$; we may create a third path by considering the path in $P_j$ from $x$ to $z$ followed by the edge from $z$ to $y$.
If $k>j+1$, then let $z^0$ be a neighbor of $x$ on $P_j$ not connected by a diagonal.
There exists a common neighbor $t$ of $z^0$ and $x$ on $P_1$; let $z^1$ be the other neighbor of $z^0$ on $P_1$.
We may choose $z^2$ to be the common neighbor of $z^1$ and $v^1$ on $P_2$.
Continue in this fashion, choosing $z^m$ to be the common neighbor of $z^{m-1}$ and $v^{m-1}$ on $P_m$, until one reaches $z^{k-1}$ on $P_{k-1}$.
If neither $v^{k-j}$ nor $w^{k-j}$ are equal to $y$, choose a neighbor $s$ of $y$ on $P_{k-1}$ such that $s$ is not $v^{k-1}$ or $w^{k-1}$.
Extend the path $x,z^1,z^2,\ldots,z^{k-1}$ to $s$ by traversing $P_{k-1}$, then connect to $y$.
If one of $v^{k-j}$ or $w^{k-j}$ is equal to $y$, then extend $z^{k-1}$ to $z^k$ on $P_k$ and connect $z^k$ to $y$ on $P_k$ to complete the path.
Our result is a third path that is disjoint from the first two, connecting $x$ to $y$.
Thus, $G_n$ is $3$-connected and planar, hence the one-skeleton of a $3$-dimensional polytope.


\section{Invariant subcomplexes}\label{Sphere2}

In \cite{BraunAutStableKneser}, the first author proved that for $k\geq 1$ and $n\geq 1$ the automorphism group of $SG_{n,k}$ is isomorphic to the dihedral group of order $2(2n+k)$, which we denote $D_{2n+k}$.
This action arises naturally, as $D_{2n+k}$ acts on $[2n+k]$ thought of as a regular $(2n+k)$-gon with vertices labeled cyclically; this action preserves stable $n$-sets and disjointness, hence induces an action on $SG_{n,k}$.
It is clear from the example in Figure~\ref{p} that this action does not restrict to simplicial automorphisms of $\mNt$, because the vertices $\{1,3,5\}$ and $\{3,5,7\}$ are in the same $D_{2n+2}$-orbit but do not have simplicially isomorphic neighborhoods.
In general, the vertices $\{1,3,5,\ldots,2n-1\}$ and $\{3,5,7,\ldots,2n+1\}$ share this behavior.
It is interesting to search for a polytopal boundary sphere contained in $\mN(SG_{n,k})$ that is invariant under this group action.
In the case $k=2$, we can find such a sphere after passing to a partial subdivision.


\subsection{Subdividing and collapsing $\mN(SG_{n,2})$}\label{subdivide}

We subdivide $\mN(SG_{n,2})$ into a complex we call $\mNh$ by leaving the facets of loose vertices unchanged and subdividing only the facets of tight vertices.
We shall consider the case where $\alpha$ is a tight vertex consisting of even elements.
For any such $\Sigma_\alpha$, $n+1$ of its vertices are the even, tight vertices and the remaining vertex is a loose vertex consisting of $n-1$ even elements and one odd element.
Order the tight even sets in $\mN(SG_{n,2})$ lexicographically, denoted by $\alpha^1, \ldots , \alpha^{n+1}$, and label the loose set in $\Sigma_{\alpha^i}$ by $\eta_{\alpha^i}$.

For the facet $\Sigma_{\alpha^i}$, using the notation of Subsection~\ref{collapse}, we have distinguished vertices $v^j$ and $v^{j+1}$.
Note that $(v^j \cap v^{j+1}) \subset \eta_{\alpha^i}$.
Recall that since each of these facets $\Sigma_{\alpha^i}$ contain all the odd, tight vertices, they intersect in a common $n$-dimensional face which we will denote by $F_o$.
Barycentrically subdivide $F_o$, and subdivide $\Sigma_{\alpha^i}$ by coning over the subdivision of $F_o$ with $\eta_{\alpha^i}$.
To form $\mNh$, apply this subdivision and an identical procedure to the odd tight vertices; denote by $F_e$ the $n$-dimensional face given by the even, tight vertices.

The complex we collapse onto will arise as a subcomplex of $\mNh$, which we denote $M(SG_{n,2})$.  
We will first produce a simplicial collapsing on $\mN(SG_{n,2})$ that preserves $F_o$ and $F_e$, then subdivide the $F$'s and some adjacent cells, and finally complete the collapsing on this subdivided complex.
Our strategy is very similar to the one used to create $\mNt$, and consists of the following steps:
\begin{enumerate}
\item In facets of loose vertices we collapse on the free edge formed by the outer neighbors of the vertex.
\item In each $\Sigma_{\alpha^i}$, we collapse the faces of $\Sigma_{\alpha^i}$ containing $\eta_{\alpha^i}$, except for the triangle $\{v^j,v^{j+1},\eta_{\alpha^i}\}$.
\item On the $F$'s, we barycentrically subdivide and then collapse all faces except the triangles $\{\{v^i,v^{i+1}\},v^i,b\}$ and $\{\{v^i,v^{i+1}\},v^{i+1},b\}$, where $b$ is the barycenter of $F$ and the $v^i$'s are the same notation introduced in Subsection~\ref{collapse}.  We also subdivide the triangles $\{v^j,v^{j+1},\eta_{\alpha^i}\}$ by subdividing the edge $\{v^j,v^{j+1}\}$.
\end{enumerate}
Via these collapses, the facets of loose simplices will collapse to our previous pairs of triangles sharing an edge, while the union of the subdivided $\Sigma_{\alpha^i}$'s will deformation retract to complexes given as a barycentrically subdivided polygon with a triangle glued to each boundary edge.

For the first two steps of our process, we use the poset map $\Phi$ from Subsection~\ref{collapse}, and apply the matchings used there on the preimages $\Phi^{-1}(B_\alpha)$ and $\Phi^{-1}(C_\alpha)$ ranging over all stable $n$-sets $\alpha$.
The resulting matching induces a simplicial collapse onto a subcomplex of $\mN(SG_{n,2})$ consisting of a pair of triangles for each loose vertex, the simplices $F_o$ and $F_e$, and a triangle of the form $\{v^j,v^{j+1},\eta_{\alpha^i}\}$ for each tight set $\alpha^i$.

For the third step in our process, we will subdivide and collapse $F_o$ and $F_e$, along with the $\{v^j,v^{j+1},\eta_{\alpha^i}\}$ triangles.
We illustrate this only for $F_o$; $F_e$ is handled identically.
Label the odd, tight stable $n$-sets  $v^1, \ldots, v^{n+1}$ as before.
Apply this labeling to $F_o$; barycentrically subdivide $F_o$, relabeling the remaining vertices in the standard way except we use the label of $b$ for the barycenter.
To ensure that our subdivision remains a simplicial complex, we must also subdivide each $\{v^j,v^{j+1},\eta_{\alpha^i}\}$  into two triangles, $\{\{v^j,v^{j+1}\},v^j,\eta_{\alpha^i}\}$ and $\{\{v^j,v^{j+1}\},v^{j+1},\eta_{\alpha^i}\}$.

Let $\Psi$ be the poset map from the face poset of $F_o$ to the $2$-chain $Q:=0<1$ such that
\[
\Psi(x) = \left\{ \begin{array}{ll}
0 & \mbox{ if } x\subseteq\{\{v^m,v^{m+1}\}, m, b\}  \mbox{ or } x\subseteq\{\{v^m,v^{m+1}\},v^{m+1}, b\} \\
1_o & \mbox{ if } \mbox{ otherwise }\\
\end{array}
\right.
\]
where $m\in [n+1]$ and $m+1 = 1$ if $m=n+1$.
For $x\in \Psi^{-1}(1)$, we match $x$ with $x \cup b$ if $b \notin x$.
This matching is clearly acyclic, and it is perfect since if $w\in \Psi^{-1}(0)$ and $b\notin w$, then $b\cup w$ is contained in $\Psi^{-1}(0)$.

It is straightforward to paste the $\Psi$-maps for $F_o$ and $F_e$ together into a single poset map into the poset consisting of two $2$-element chains sharing a common minimal element, i.e. $\{0,1_o,1_e\}$ such that $0<1_o$ and $0<1_e$.
If a face of our collapsed, then subdivided, complex from the first two steps is not mapped by $\Psi$ for $F_o$ or $F_e$, then map it to $0$.
The resulting poset map allows the application of Theorem~\ref{patchwork}.

The proof that $M(SG_{n,2})$ is the boundary of a $3$-dimensional polytope is almost identical to the proof in Subsection~\ref{polytopeproof}.
One only needs to observe that the complex resulting from the current analysis is obtained from our previous case by removing the edges inside $P_0$ and $P_n$, barycentrically subdividing $P_0$ and $P_n$, and then subdividing the remaining triangles sharing an edge with $P_0$ or $P_n$.
The proof that the one-skeleton is $3$-connected is the same aside from handling the situation where vertices arising from the subdivision are involved, which is a straightforward modification of the argument given in the previous case.


\subsection{Action of $D_{2n+2}$ on $M(SG_{n,2})$}

Our goal in this subsection is to show that $D_{2n+2}$ acts simplicially on $M(SG_{n,2})$.
Consider $[2n+2]$ as the set of vertices of a regular $(2n+2)$-gon on which $D_{2n+2}$ acts in the usual way.


Let $\alpha = \{\alpha_1 , \ldots, \alpha_n\}$ be a loose stable set with $\alphaodd = \{\alpha_1, \ldots, \alpha_i\}$ and
$\alphaeven = \{\alpha_{i+1}, \ldots, \alpha_n\}$
Let the immediate neighbors of $\alpha$ be denoted $i(1):=\alpha\oplus1$ and $i(2):=\alpha\ominus1$.
Let the outer neighbors of $\alpha$ be denoted
\begin{eqnarray*}
o(1) &:=& \{\alpha_1 + 1 , \ldots , \alpha_{i-1} +1 , \alpha_i + 2, \alpha_{i+1} +1 , \ldots , \alpha_n +1\}  \\
o(2) &:=& \{\alpha_1 +1 , \ldots , \alpha_{n-1} +1 , \alpha_n + 2\}.
\end{eqnarray*}
There are two simplices associated to $\alpha$ in $M(SG_{n,2})$, given by $\{i(1),i(2),o(j)\}$ for $j=1,2$.
As $D_{2n+2}$ is the automorphism group of $SG_{n,2}$, the neighbors of $\alpha$ are mapped by an element $g\in D_{2n+2}$ to neighbors of $g(\alpha)$.
Since $\alpha$ is a loose set and $D_{2n+2}$ clearly preserves the loose and tight conditions, $g(\alpha)$ is also a loose set.
We will show that the outer neighbors of $\alpha$ are carried by $g$ to the outer neighbors of $g(\alpha)$, hence each of these two simplices associated to $\alpha$ are taken to one of the simplices associated to $g(\alpha)$.

For an element $g\in D_{2n+2}$, $g$ is either a rotation or a flip of $[2n+2]$.
If $g$ is a rotation, then
\begin{eqnarray*}
g(o(1)) &=& \{g(\alpha_1 + 1) , \ldots , g(\alpha_{i-1} +1) , g(\alpha_i + 2), g(\alpha_{i+1} +1) , \ldots , g(\alpha_n +1)\}, \\
		&=& \{g(\alpha_1) + 1 , \ldots , g(\alpha_{i-1}) +1 , g(\alpha_i) + 2, g(\alpha_{i+1}) +1 , \ldots , g(\alpha_n) +1\}.
\end{eqnarray*}
Otherwise, $g$ is a flip and
\begin{eqnarray*}
g(o(1)) &=& \{g(\alpha_1 + 1) , \ldots , g(\alpha_{i-1} + 1) , g(\alpha_i + 2), g(\alpha_{i+1} +1) , \ldots , g(\alpha_n + 1)\} \\
		&=& \{g(\alpha_1) - 1 , \ldots , g(\alpha_{i-1}) -1 , g(\alpha_i) - 2, g(\alpha_{i+1}) -1 , \ldots , g(\alpha_n) -1\} \\
		&=& \{g(\alpha_1) + 1 , \ldots , g(\alpha_{i-1}) +1 , g(\alpha_i) +1 , g(\alpha_{i+1}) +1 , \ldots , g(\alpha_n) +2\}
\end{eqnarray*}
In either case, $g(o(1))$ is an outer neighbor of $\alpha$ and by a similar argument, $g(o(2))$ is an outer neighbor of $\alpha$.
Thus $D_{2n+2}$ sends the associated simplices in $M(SG_{n,2})$ of a loose set $\alpha$ to the associated simplices of $g(\alpha)$.


Let $\alpha$ be a tight stable set and consider the four triangles $T_1:=\{v^j, \{v^j,v^{j+1}\}, b\}$, $T_2:=\{v^{j+1}, \{v^j,v^{j+1}\}, b\}$, $T_3:=\{v^j, \{v^j,v^{j+1}\}, \eta_{\alpha}\}$, and $T_4:=\{v^{j+1}, \{v^j,v^{j+1}\}, b\}$ in $M(SG_{n,2})$ where $v^{j+1}$ is a neighboring tight vertex, $\{v^j,v^{j+1}\}$ is the barycenter of the edge $v^j v^{j+1}$, $\eta_{\alpha}$ is the unique vertex that is both a stable set and a neighbor of $v^j$ and $v^{j+1}$ in $M(SG_{n,2})$, and $b$ is the barycenter of the $(n+1)$-gon.
Every remaining facet in $M(SG_{n,2})$ can be associated to a tight stable set $\alpha$ in this way, e.g. every remaining facet contains some tight stable set as a vertex.
We now apply $g$ to these triangles and show that their image is contained in $M(SG_{n,2})$.

For $T_1$, we have $g(T_1) = \{ \{g(v^j), g(\{v^j,v^{j+1}\}), g(b)\}$.
As $g$ is either parity preserving or reversing for all elements of $[2n+2]$, we know that $g(b)$ is either $b$ or the corresponding element of opposite parity.
In either case, $g(v^j)$ and $g(b)$ are neighbors in our complex as well as $g(\{v^j,v^{j+1}\})$ and $g(b)$.
Finally, $g(\{v^j, v^{j+1}\})$ and $g(v^j)$ are neighbors in our complex as $g(\{v^j, v^{j+1}\})= \{g(v^j), g(v^{j+1})\}$.  So $T_1$ (as well as $T_2$ by symmetry) maps to a corresponding triangle in our complex for any $g\in D_{2n+2}$.

To see that $T_3$ and $T_4$ map to appropriate triangles, we need only check that $g(\eta_{\alpha})$ is a neighbor of $g(v^j)$ and $g(v^{j+1})$ in our complex.
By definition and construction the set $\eta_{\alpha} = \{v^j\}\cap \{v^{j+1}\} \cup \{p\}$ where $p$ is the unique element of opposite parity of the elements of $v^j$ and $v^{j+1}$ that allows $\eta_{\alpha}$ to remain stable.
Then $g(\eta_{\alpha}) = g(v^j) \cap g(v^{j+1}) \cup g(p)$.
As $g(v^j)$ and $g(v^{j+1})$ are connected via $g(\{v^j,v^{j+1}\})$, we know that they have exactly $n-1$ elements in common.
Moreover, $g(p)$ is of opposite parity of the elements of $g(v^j)$.
Hence, $g(\eta_{\alpha})$ is stable and a neighbor of both $g(v^j)$ and $g(v^{j+1})$ in $M(SG_{n,2})$, thus $D_{2n+2}$ preserves triangles of this form as well.


\bibliographystyle{plain}
\bibliography{Braun}

\begin{thebibliography}{10}

\bibitem{BabsonKozlovComplexes}
Eric Babson and Dmitry~N. Kozlov.
\newblock Complexes of graph homomorphisms.
\newblock {\em Israel J. Math.}, 152:285--312, 2006.

\bibitem{BabsonKozlovLovaszConjecture}
Eric Babson and Dmitry~N. Kozlov.
\newblock Proof of the {L}ov\'asz conjecture.
\newblock {\em Ann. of Math. (2)}, 165(3):965--1007, 2007.

\bibitem{BjornerDeLongueville}
Anders Bj\"{o}rner and Mark de~Longueville.
\newblock Neighborhood complexes of stable {K}neser graphs.
\newblock {\em Combinatorica}, 23(1):23--34, 2003.

\bibitem{BraunIndComplexKneser}
Benjamin Braun.
\newblock Independence complexes of stable {K}neser graphs.
\newblock submitted, 2010.

\bibitem{BraunAutStableKneser}
Benjamin Braun.
\newblock Symmetries of the stable {K}neser graphs.
\newblock {\em Advances in Applied Mathematics}, 45(1):12 -- 14, 2010.

\bibitem{DochtermannEngstromCellular}
Anton Dochtermann and Alexander Engstr{\"o}m.
\newblock Cellular resolutions of cointerval ideals.
\newblock To appear in Mathematische Zeitschrift, DOI:
  10.1007/s00209-010-0789-z. Preprint at http://arxiv.org/abs/1004.0713.

\bibitem{DochtermannSchultz}
Anton Dochtermann and Carsten Schultz.
\newblock Topology of {H}om complexes and test graphs for bounding chromatic
  number.
\newblock to appear in Israel Journal of Math, 2010,
  http://arxiv.org/abs/0907.5079.

\bibitem{FormanMorseTheory}
Robin Forman.
\newblock Morse theory for cell complexes.
\newblock {\em Adv. Math.}, 134(1):90--145, 1998.

\bibitem{JonssonBook}
Jakob Jonsson.
\newblock {\em Simplicial complexes of graphs}, volume 1928 of {\em Lecture
  Notes in Mathematics}.
\newblock Springer-Verlag, Berlin, 2008.

\bibitem{KozlovBook}
Dmitry Kozlov.
\newblock {\em Combinatorial algebraic topology}, volume~21 of {\em Algorithms
  and Computation in Mathematics}.
\newblock Springer, Berlin, 2008.

\bibitem{LovaszChromaticNumberHomotopy}
L.~Lov{\'a}sz.
\newblock Kneser's conjecture, chromatic number, and homotopy.
\newblock {\em J. Combin. Theory Ser. A}, 25(3):319--324, 1978.

\bibitem{Schrijvergraphs}
A.~Schrijver.
\newblock Vertex-critical subgraphs of {K}neser graphs.
\newblock {\em Nieuw Arch. Wisk. (3)}, 26(3):454--461, 1978.

\bibitem{SchultzStableKneserNotTest}
Carsten Schultz.
\newblock The equivariant topology of stable {K}neser graphs.
\newblock preprint, http://arxiv.org/abs/1003.5688.

\bibitem{SchultzSpacesOfCircuits}
Carsten Schultz.
\newblock Graph colorings, spaces of edges and spaces of circuits.
\newblock {\em Adv. Math.}, 221(6):1733--1756, 2009.

\bibitem{ZieglerLectures}
G{\"u}nter~M. Ziegler.
\newblock {\em Lectures on polytopes}, volume 152 of {\em Graduate Texts in
  Mathematics}.
\newblock Springer-Verlag, New York, 1995.

\end{thebibliography}

\end{document}